\documentclass[11pt]{article}
\usepackage{enumerate, amsmath, amsthm, amsfonts, amssymb, xy}
\usepackage[margin=1in]{geometry}
\usepackage{hyperref}
\usepackage{algorithm, algorithmic}
\usepackage{color}

\input xy
\xyoption{all}

\newcommand{\bm}[1]{{\mbox{\boldmath $#1$}}}

\newtheorem{theorem}{Theorem}[section]
\newtheorem{proposition}[theorem]{Proposition}
\newtheorem{lemma}[theorem]{Lemma}
\newtheorem{corollary}[theorem]{Corollary}

\numberwithin{equation}{section}

\theoremstyle{remark}
\newtheorem{rmk}[theorem]{Remark}
\theoremstyle{definition}
\newtheorem{definition}[theorem]{Definition}
\newtheorem{example}[theorem]{Example}

\title{On the Number of Fixed-Length Semiorders}

\author{Yangzhou Hu\footnote{The author's research was part of an undergraduate research project at M.I.T. under the supervision of Richard Stanley.} \\Department of Mathematics\\ Massachusetts
  Institute of Technology\\yangzhou@mit.edu} 

\begin{document}

\maketitle

\begin{abstract}
A \emph{semiorder} is a partially ordered set $P$ with two certain
forbidden induced subposets.  This paper establishes a bijection between $n$-element semiorders of 
length $H$ and $(n+1)$-node ordered trees of height $H+1$. This
bijection preserves not only the number of elements, but also much
additional structure. Based on this correspondence, we calculate the
generating functions and explicit formulas for the numbers of labeled
and unlabeled $n$-element semiorders of length $H$.  We also prove
several concise recurrence relations and provide combinatorial proofs
for special cases of the explicit formulas.
\end{abstract}

\section{Introduction and Main Theorem}\label{ch1}
We will use partially ordered set (poset) notation and terminology
from \cite[Ch.~3]{Stanley}.
A \emph{semiorder}  is a poset without the
following induced subposets:
\begin{itemize}
\item $(\bm{2}+\bm{2})$: four distinct elements $x, y, z, 
  w$, such that $x>y, z>w$, and other pairs are incomparable.     
\item$(\bm{3}+\bm{1})$: four distinct elements $x, y, z,  
  w$, such that $x>y>z$, and other pairs are incomparable.    
\end{itemize}
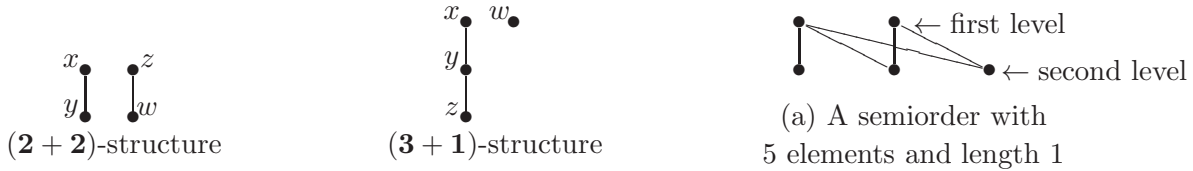
\begin{center}
$ \xy 0;/r.30pc/:
(0,0)*{\bullet}="a";
(0,5)*{\bullet}="b";
(5,0)*{\bullet}="c";
(5,5)*{\bullet}="d";
(-1.5,1)*{y};
(-1.5,6)*{x};
(6.5,1)*{w};
(6.5,6)*{z};
(3,-3)*{\text{$({\bm{2}+\bm{2}})$-structure}};
(40,0)*{\bullet}="a2";
(40,5)*{\bullet}="b2";
(40,10)*{\bullet}="c2";
(45,10)*{\bullet}="d2";
(38.5,1)*{z};
(38.5,6)*{y};
(38.5,11)*{x};
(43.5,11)*{w};
(43,-3)*{\text{$({\bm{3}+\bm{1}})$-structure}};
(75,10)*{\bullet}="b3";
(85,10)*{\bullet}="c3";
(95,10)*{\leftarrow\text{first level}};
(75,5)*{\bullet}="d3";
(85,5)*{\bullet}="e3";
(95,5)*{\bullet}="f3";
(106,5)*{\leftarrow\text{second level}};
(87, 0)*{\text{(a) A semiorder with}};
(87, -4)*{\text{5 elements and length 1}};
"b3";"d3"**\dir{-};
"b3";"e3"**\dir{-};
"b3";"f3"**\dir{-};
"c3";"e3"**\dir{-};
"c3";"f3"**\dir{-};
 "a"; "b"**\dir{-};
 "c"; "d"**\dir{-};
 "a2"; "b2"**\dir{-};
 "b2"; "c2"**\dir{-};
\endxy$
\end{center}

In other words, semiorders are $(\bm{2}+\bm{2})$-free and
$(\bm{3}+\bm{1})$-free posets.  Every semiorder can also be regarded
as a partial ordering $P$ of a subset of $\mathbb{R}$ defined by $x<y$
in $P$ if $x<y-1$ in $\mathbb{R}$.  The \emph{length} $H$ of a
semiorder is the length of a longest chain.  Every semiorder $R$ with
$n$ elements, up to isomorphism, can be uniquely represented as an
integer vector $\rho(R)=(r_1, r_2, \dots, r_n)$, where $r_i$
represents the number of elements smaller than the $i^{\mathrm{th}}$ element,
and $r_1\geq r_2 \geq \dots \geq r_{n}\geq 0$, $r_i \leq n-i$, for all
$1\leq i \leq n$.  For instance, the above graph (a) presents a
semiorder $R$ with $5$ elements, length $1$, and vector
$\rho(R)=(3,2,0,0,0)$.  For further basic information on semiorders,
see \cite{linearextension}.

There is much interest in enumerating the number of posets with
certain properties.  For example, Bousquet-M\'{e}lou et al. enumerated the
number of $(\bm{2}+\bm{2})$-free posets \cite{2+2free}.  It is a
classical result of Wine and Freund \cite{Wine} that the number of
nonisomorphic $n$-element semiorders is the Catalan number
$C_n=\frac{1}{n+1} {2n\choose n}$, while Chandon, Lemaire, and Pouget
\cite{Chandon} showed (in an equivalent form) that if $f(n)$ is the
number of $n$-element \emph{labeled} semiorders (or semiorders on an
$n$-element set), then $\sum_{n \geq 0} f(n)\frac{x^n}{n!}=\sum_{n\geq
  0} C_n(1-e^{-x})^n$.  For a general principle implying this result,
see Lemma~\ref{Yan}.  In this paper, we deal with semiorders of length
at most $H$.  That is, we enumerate the number of posets which are
$(\bm{2}+\bm{2})$-free, $(\bm{3}+\bm{1})$-free, and of length at most
$H$.  We carry out the enumeration by establishing a bijection between
semiorders and ordered trees of a fixed height. 

An \emph{ordered tree} is a rooted tree that has been embedded in the
plane so that the relative order of subtrees at each node is part of
its structure.  The \emph{height} $H$ of an ordered tree is the number
of edges in a chain of maximum length.  The following graph (b) shows
an ordered tree with 6 nodes and height 2.  

$$ 
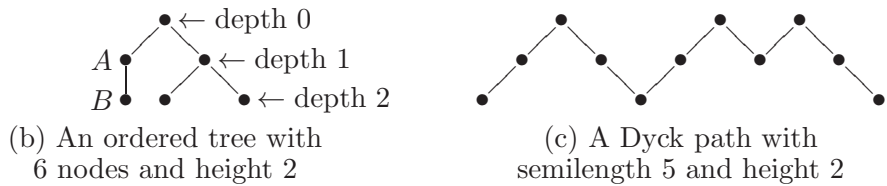
\xy 0;/r.25pc/: 
(80,-10)*{\bullet}="a2";
(85,-5)*{\bullet}="b2";
(90,0)*{\bullet}="c2";
(95,-5)*{\bullet}="d2";
(100,-10)*{\bullet}="e2";
(105,-5)*{\bullet}="f2";
(110,0)*{\bullet}="g2";
(115,-5)*{\bullet}="h2";
(120,0)*{\bullet}="i2";
(125,-5)*{\bullet}="j2";
(130,-10)*{\bullet}="k2";
(105, -15)*{\text{(c) A Dyck path with}};
(105, -19)*{\text{semilength 5 and height 2}};
"a2";"b2"**\dir{-};
"b2";"c2"**\dir{-};
"c2";"d2"**\dir{-};
"d2";"e2"**\dir{-};
"e2";"f2"**\dir{-};
"f2";"g2"**\dir{-};
"g2";"h2"**\dir{-};
"h2";"i2"**\dir{-};
"i2";"j2"**\dir{-};
"j2";"k2"**\dir{-};
(40,0)*{\bullet}="a3";
(50,0)*{\leftarrow\text{depth 0}};
(35,-5)*{\bullet}="b3";
(32,-5)*{A};
(45,-5)*{\bullet}="c3";
(55,-5)*{\leftarrow\text{depth 1}};
(35,-10)*{\bullet}="d3";
(32,-10)*{B};
(40,-10)*{\bullet}="e3";
(50,-10)*{\bullet}="f3";
(60,-10)*{\leftarrow\text{depth 2}};
(40, -15)*{\text{(b) An ordered tree with}};
(40, -19)*{\text{6 nodes and height 2}};
"a3";"b3"**\dir{-};
"a3";"c3"**\dir{-};
"b3";"d3"**\dir{-};
"c3";"e3"**\dir{-};
"c3";"f3"**\dir{-};
\endxy$$

A \emph{Dyck path} of semilength $n$ is a lattice path in the
Euclidean plane from $(0,0)$ to $(2n,0)$ whose steps are either
$(1,1)$ or $(1,-1)$ and the path never goes below the $x$-axis.  The
height $H$ of a Dyck path is the maximal $y$-coordinate among all
points on the path.  The above graph (c) shows a Dyck path with
semilength 5 and height 2.

It is well-known that there is a one-to-one correspondence between (i)
ordered trees with $n+1$ nodes and height $H$ and (ii) Dyck paths
with semilength $n$ and height $H$.  This paper establishes a
bijection between $n$-element semiorders of length $H$ and
$(n+1)$-node ordered trees of height $H+1$. Thus these semiorders
simultaneously correspond to Dyck paths of semilength $n$ and height
$H+1$.

\begin{theorem} [Main Theorem] \label{main} 
For $n\geq 1$ and $H\geq 0$, the number of nonisomorphic $n$-element
unlabeled semiorders of length $H$ is equal to the number of
$(n+1)$-node ordered trees of height $H+1$, which is also equal to the
number of Dyck paths with semilength n and height $H+1$. 
\end{theorem}

Section \ref{sec2} gives a recurrence proof and a bijective proof
for Theorem \ref{main}.  Section~\ref{sec3} calculates the
generating functions and explicit formulas for the number of unlabeled
as well as labeled semiorders with fixed lengths.  Section~\ref{sec5}
proves some concise recurrence relations, and Section~\ref{sec6} shows
explicit formulas for the number of semiorders of
certain lengths $H$, and provides simple bijective proofs for these
formulas.   

\section{Proof of Main Theorem ~\ref{main}}\label{sec2}
Before proving Theorem~\ref{main}, we first define some terminology
that is used later in the proof. 

\begin{definition}\label{treelevel}
A node $A$ in an ordered tree has \emph{depth} $i$ if the distance
from node $A$ to the root is $i$. In particular, the depth of the root is
$0$.  
\end{definition}

In this paper we regard the root as the uppermost node, and all other
nodes are below the root.  We say node $B$ is attached to node $A$ if
node $B$ has depth $i+1$ and node $A$ has depth $i$, and these two
nodes are adjacent.  Refer to graph (b) as an example. 

\begin{definition}\label{2.3}
An element $a$ of a semiorder is on the $i^{\mathrm{th}}$
$\emph{level}$ ($i \geq 1$) if $i$ is the largest integer for which
there exist $i-1$ elements $a_1, a_2, \dots, a_{i-1}$ satisfying $a_1
> a_2 >\dots > a_{i-1} >a$.  Refer to graph (a) as an example.
\end{definition}

\begin{proposition}\label{marker}
For a length $H$ semiorder, and for $1\leq i \leq H$, there is at
least one element on the $i^{\mathrm{th}}$ level that is larger than all
elements on the $(i+1)^{\mathrm{th}}$ level. 
\end{proposition}

\begin{proof}
Suppose that there does not exist an element on the $i^{\mathrm{th}}$ level
that is larger than all elements on the $(i+1)^{\mathrm{th}}$ level. Since
every element on the $(i+1)^{\mathrm{th}}$ level must be smaller than at least
one element on the $i^{\mathrm{th}}$ level, there must exist two elements $a$
and $c$ on the $(i+1)^{\mathrm{th}}$ level which are smaller than two distinct
elements $b$ and $d$ on the $i^{\mathrm{th}}$ level, respectively, and $b$ is
not larger than $c$, while $d$ is not larger than $a$.  Then $\{b>a$,
$d>c\}$ forms a $(\bm{2}+\bm{2})$-structure,  a contradiction.
Therefore, at least one element on the $i^{\mathrm{th}}$ level is larger than
all elements on the $(i+1)^{\mathrm{th}}$ level.  
\end{proof}

\begin{proposition}\label{2}
For a length $H$ semiorder, and for $1\leq i < j \leq H+1$, $j-i\geq
2$, every element on the $i^{\mathrm{th}}$ level is larger than all
elements on the $j^{\mathrm{th}}$ level.
\end{proposition}

\begin{proof}
Assume to the contrary that there exist an element $b$ on the
$i^{\mathrm{th}}$ level and an element $a$ on the $j^{\mathrm{th}}$
level such that $b$ is not larger than $a$.  By Definition \ref{2.3},
there exist $j-1$ elements $a_1, a_2, \dots, a_{j-1}$, such that $a_1
> a_2 >\dots > a_{j-1} >a$, and then $a_{j-2}$ and $a_{j-1}$ should be
on the $(j-2)^{\mathrm{th}}$ and $(j-1)^{\mathrm{th}}$ level,
respectively.  Since element $b$ is on the $i^{\mathrm{th}}$ level,
and $j-i\geq 2$, element $b$ cannot be smaller than any of $a_{j-2},
a_{j-1}$, or $a$.  In addition, since $b$ is not larger than $a$ and
$a_{j-2}> a_{j-1}> a$, $b$ is not comparable with any of $a_{j-2},
a_{j-1}$ or $a$.  Hence, $\{a_{j-2}> a_{j-1}> a$, $b\}$ forms a
$(\bm{3}+\bm{1})$-structure, a contradiction.  Therefore every element
on the $i^{\mathrm{th}}$ level should be larger than all elements on
the $j^{\mathrm{th}}$ level, for $j-i\geq 2$.
\end{proof}

We are now ready to prove the Main Theorem \ref{main}.  We give two
proofs here: one considers the recurrence formulas of the two numbers
in the theorem, and the other directly establishes a bijection between
semiorders and ordered trees.

\subsection{Recurrence proof}
Let $t(n,h,k)$\label{tnhk} denote the number of $(n+1)$-node ordered
trees of height $h+1$, for which exactly $k$ nodes have depth $h+1$,
$1\leq k \leq n$.  Let $f(n,h,k)$ be the number of
$n$-element semiorders of length $h$, and exactly $k$ elements are on
the last level.  We show that $t(n,h,k)$ and $f(n,h,k)$ have the same
initial value and recurrence formula in the following lemmas, and thus
they are equal.

\begin{lemma}\label{lemma2.6}
For $h\geq 1$, we have 
\begin{equation}\label{t}
t(n,h,k)=\sum_{m=1}^{n-k}\binom{m+k-1}{m-1}\cdot t(n-k,h-1,m).
\end{equation}
\end{lemma}

\begin{proof}
Say we have an $(n-k+1)$-node ordered tree of height $h$, and assume
that exactly $m$ nodes have depth $h$, $1\leq m \leq n-k$.  Consider
adding $k$ nodes to the tree to get a new tree with $n+1$ nodes and
height $h+1$, and the newly added nodes are exactly the set of nodes
of depth $h+1$.  Thus we need to attach the $k$ new nodes to the $m$
nodes of depth $h$, and every new node is uniquely attached to one
node.  Let the $m$ nodes be $A_1, A_2, \dots, A_m$, and the number of
new nodes attached to $A_i$ be $r_i$, $1 \leq i \leq m$.  Then we have
$r_1 + r_2 + \dots + r_m = k, \ \ r_i\geq 0, \hspace{6pt}1 \leq i \leq
m$. 

The number of integer solutions to the above equation is
$\binom{m+k-1}{m-1}$.  Therefore, we have $\binom{m+k-1}{m-1}$ ways to
add the $k$ nodes.  Summing up all possible $m$'s, we obtain 
$$t(n,h,k)=\sum_{m=1}^{n-k}\binom{m+k-1}{m-1}\cdot t(n-k,h-1,m).$$
\end{proof}

\begin{lemma}\label{lemma2.7}
For $h\geq 1$, we have 
\begin{equation}\label{f}
f(n,h,k)=\sum_{m=1}^{n-k}\binom{m+k-1}{m-1}\cdot f(n-k,h-1,m).
\end{equation}
\end{lemma}

\begin{proof}
We say that an element of a semiorder is \emph{good} if the element is
on the last level of the semiorder.  Say we have an $(n-k)$-element
semiorder $S$ of length $h-1$ and $m$ good elements, $1\leq m \leq
n-k$.  Consider adding $k$ elements to $S$ to get a new semiorder $S'$
with $n$ elements and length $h$, and the newly added elements are
exactly the set of good elements of $S'$. Call the original $m$ good
elements $a_1, a_2, \dots, a_m$, and the $k$ new elements $b_1, b_2,
\dots, b_k$.  Then in the semiorder $S'$, we have that $a_1, a_2,
\dots, a_m$ are the only elements on the $h^{\mathrm{th}}$ level, and
$b_1, b_2, \dots, b_k$ are the only elements on the
$(h+1)^{\mathrm{th}}$ level.

If we remove all elements on the first $h-1$ levels of $S'$, then we
get a length one semiorder $P$ with $m+k$ elements, and there are
exactly $m$ elements on the upper level and $k$ elements on the lower
level.  On the other hand, given a semiorder $S$ and a semiorder $P$
as above, we can uniquely determine the semiorder $S'$, because based
on Proposition \ref{2}, the $k$ elements on the $(h+1)^{\mathrm{th}}$
level of $S'$ must be smaller than all elements on the
$i^{\mathrm{th}}$ level of $S'$, for $1 \leq i \leq h-1$.  Therefore,
the semiorder $P$ uniquely determines the way to add the $k$ new
elements.

Let $P$ with $\rho(P)=(p_1, p_2, \dots, p_{m+k})$ represent one such
semiorder.  Then we have  
\begin{equation}\label{newsemiorder}
\begin {cases} k=p_1 \geq p_2 \geq \dots \geq p_m \geq 0 \\ p_{m+1} =
  p_{m+2} = \dots = p_{m+k} = 0. \end{cases} 
\end{equation}

Notice that $\{p_2, p_3, \dots, p_{m}\}$ is an $(m-1)$-element
multiset with elements from $\{0, 1, \dots, k\}$, and thus we have
$\binom{k+1+m-1-1}{m-1} = \binom{m+k-1}{m-1}$ such
multisets. Therefore, there are $\binom{m+k-1}{m-1}$ possible
semiorder $P$'s.  Summing up all possible $m$'s, we have  
$$f(n,h,k)=\sum_{m=1}^{n-k}\binom{m+k-1}{m-1}\cdot f(n-k,h-1,m).$$
\end{proof}

\begin{proof} [Proof of the Main Theorem \ref{main}. ]

For $h=0$, the $(n+1)$-element ordered tree of height $h+1=1$ can only
be the tree with $n$ nodes adjacent to the root; meanwhile, the
$n$-element semiorder of length $0$ can only be the one with $n$
elements and any two of the elements are incomparable.  As a result,
we have 
$$t(n,0,k) = f(n,0,k)= \begin{cases} 1 &\text{if $n=k$}\\ 0 & \text{if
    $n \neq k.$}\end {cases}$$ 

For $h\geq1$, by Lemma \ref{lemma2.6} and \ref{lemma2.7},
$t(n,h,k)$ and $f(n,h,k)$ have the same recurrence formula.
Therefore $t(n,h,k) = f(n,h,k)$ for every $n\geq 1$, $ h\geq 0$, and
$1\leq k \leq n$.  Summing on $k$ completes the proof of Theorem
\ref{main}. 
\end{proof}

\vspace{6pt}
\subsection{Bijective proof}
Recall that an element of a semiorder is good if it is on the last
level of the semiorder.  Based on the idea in the recurrence proof, we
can construct a one-to-one map from $(n+1)$-element ordered trees of height
$H+1$ with $k$ nodes of depth $H+1$ to $n$-element semiorders of length $H$
with $k$ good elements.   

For an ordered tree with $n+1$ nodes and height $H+1$, let us assume that there are $x_i$
nodes of depth $i$, $ 0\leq i \leq H+1$.  Since the
root is the only node of depth $0$, we have
\begin{equation}\label{sum1} 
\sum_{i=1}^{H+1} x_i = n.
\end{equation}

Let $s_j^i$ denote the number of nodes of depth $i$
that are adjacent to the $j^{\mathrm{th}}$ node of depth
$i-1$, $1\leq j \leq x_{i-1}$, $1\leq i \leq
H+1$.  Since every node of depth $i$ should be
adjacent to exactly one node of depth $i-1$, we
must have
\begin{equation}\label{sum2}
\sum_{j=1}^{x_{i-1}}s_j^i = x_i.
\end{equation}

Let $u_j^i=\sum_{k=j}^{x_{i-1}}s_k^i$, $1\leq j \leq x_{i-1}$, $1\leq
i \leq H+1$. Then $u_1^i \geq u_2^i \geq \dots \geq u_{x_{i-1}}^i$.
Let $y_{i} = \sum_{k=1}^{i}x_k$, $1\leq i \leq H+1$, and $y_0=0$. We
now define $R^{i}$ by induction, and let the number of entries in
$R^{i}$ be $y_{i}$. 

Set $R^1=(0, 0, \dots, 0)$, in which there are $x_1=y_1$ zeros.
Assume $R^i = (r_1^i, r_2^i, \dots, r_{y_{i}}^i)$, $1\leq i \leq H$,
and let 
$$R^{i+1}=(r_1^i + x_{i+1}, r_2^i + x_{i+1}, \dots, r_{y_{i-1}}^i +
x_{i+1}, r_{y_{i-1}+1}^i + u_1^{i+1}, r_{y_{i-1}+2}^i + u_2^{i+1},
\dots, r_{y_{i-1}+x_{i}}^i + u_{x_{i}}^{i+1}, 0, 0, \dots, 0)$$ 
in which there are $x_{i+1}$ zeros,  and thus $R^{i+1}$ has $y_{i-1}+x_{i}+x_{i+1} =
y_{i+1}$ entries.    

\vspace{6pt}
\begin{theorem}
The vector $R^{H+1}$ represents an $n$-element semiorder of length $H$
with $x_{H+1}$ good elements.  This gives a bijective map from
$(n+1)$-node ordered trees of height $H+1$ to $n$-element semiorders
of length $H$. 
\end{theorem}
Since this map is naturally derived from the recurrence proof, we do
not give a rigorous proof on why the map is valid and why it is a
bijection.  The main idea here is to map an ordered tree of height
$H+1$ to a semiorder with $H+1$ levels, where the number of elements
on the $i^{\mathrm{th}}$ level of the semiorder is equal to the number
of nodes of depth $i$ in the tree, $1 \leq i \leq H+1$.  We get a
bijection between (a) the connections between nodes of depths $i$ and
$i+1$ in the tree, and (b) the set of ordered pairs between elements
on levels $i$ and $i+1$ in the semiorder, $1 \leq i \leq H$.  This
bijection preserves not only the number of elements but also much
additional structure.  It presents an effective way to connect
semiorders and ordered trees, as well as Dyck paths.  In order to
illustrate the bijection more clearly, we show by an example how the
map works.

\begin{example}
Assume we have the following ordered tree:
$$ 
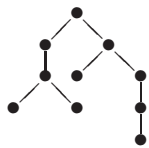
\xy 0;/r.20pc/: 
(0,0)*{\bullet}="a";
(-5,-5)*{\bullet}="b";
(5,-5)*{\bullet}="c";
(-5,-10)*{\bullet}="d";
(0,-10)*{\bullet}="e";
(10,-10)*{\bullet}="f";
(-10,-15)*{\bullet}="g";
(0,-15)*{\bullet}="h";
(10,-15)*{\bullet}="i";
(10,-20)*{\bullet}="j";
"a";"b"**\dir{-};
"a";"c"**\dir{-};
"b";"d"**\dir{-};
"c";"e"**\dir{-};
"c";"f"**\dir{-};
"d";"g"**\dir{-};
"d";"h"**\dir{-};
"f";"i"**\dir{-};
"i";"j"**\dir{-};
\endxy$$

Here the number of nodes is 10 and height is 4.  We have $(x_1, x_2, x_3, x_4)=(2, 3, 3, 1)$.   So the tree should correspond
to a semiorder with 9 elements, length 3, and the number of elements
of each depth is given by $(2, 3, 3, 1)$.  Write $s^i=(s_1^i, s_2^i,
\dots, s_{x_{i-1}}^i)$ and $u^i=(u_1^i, u_2^i, \dots, u_{x_{i-1}}^i)$. Then the vectors
representing the connections between two adjacent depth-levels in the
ordered tree are  
$$ s^1=(2), \hspace{6pt}  s^2=(1,2),\hspace{6pt}  s^3=(2, 0,
1),\hspace{6pt} s^4=(0, 0, 1).$$ 

We transform these vectors into vectors that can represent the set of
ordered pairs between two adjacent levels of the semiorder.  These
vectors are 

$$ u^1=(2),\hspace{6pt}  u^2=(3, 2),\hspace{6pt}  u^3=(3, 1,
1),\hspace{6pt} u^4=(1, 1, 1).$$ 

Now let us construct $R^i$, $1\leq i \leq 4$:
\begin{align*}
R^1=(0, 0),\hspace{6pt} R^2=(3, 2, 0, 0, 0),\hspace{6pt} R^3=(6, 5, 3,
1, 1, 0, 0, 0), \hspace{6pt}R^4=(7, 6, 4, 2, 2, 1, 1, 1, 0). 
\end{align*}

In fact, $R^i$ depicts the semiorder with only the first $i$ levels,
$1\leq i \leq 4$, and $R^4$ is the final semiorder we desired.  Its
Hasse diagram is as follows: 

$$ 
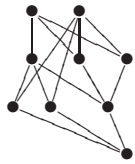
\xy 0;/r.30pc/: 
(0,0)*{\bullet}="b";
(5,0)*{\bullet}="c";
(0,-5)*{\bullet}="d";
(5,-5)*{\bullet}="e";
(10,-5)*{\bullet}="f";
(-2,-10)*{\bullet}="g";
(2,-10)*{\bullet}="h";
(8,-10)*{\bullet}="i";
(10,-15)*{\bullet}="j";
"b";"d"**\dir{-};
"b";"e"**\dir{-};
"b";"f"**\dir{-};
"c";"e"**\dir{-};
"c";"f"**\dir{-};
"c";"g"**\dir{-};
"c";"h"**\dir{-};
"d";"g"**\dir{-};
"d";"h"**\dir{-};
"d";"i"**\dir{-};
"e";"i"**\dir{-};
"f";"i"**\dir{-};
"g";"j"**\dir{-};
"h";"j"**\dir{-};
"i";"j"**\dir{-};
\endxy$$

The inverse map can be done by reversing the steps.
\end{example}

\section{Generating Functions and Explicit Formulas}\label{sec3}
\subsection{On unlabeled semiorders}
Let $f_h^n$ denote the number of nonisomorphic unlabeled semiorders
with $n$ elements and length $h$, and $f_{\leq h}^n$ denote the number
of nonisomorphic unlabeled semiorders with $n$ elements and length at
most $h$, so $f_h^n$ = $f_{\leq h}^n - f_{\leq (h-1)}^n$.  Let $F_h(x)
= \sum_{n=0}^{\infty} f_h^nx^n$, and $F_{\leq h}(x) =
\sum_{n=0}^{\infty} f_{\leq h}^nx^n.$ 

De Bruijn, Knuth, and Rice \cite{first} calculated the generating
function for the number of fixed-height ordered trees in 1972.  Based
on this generating function and Theorem \ref{main}, we have the
following corollary. 

\begin{corollary}
For $h\geq 0$, 
\begin{align}
&\bullet F_h(x) = \sum_{n=0}^{\infty}
  f_h^nx^n=\frac{x^{h+1}}{p_{h+1}(x)p_{h}(x)}\label{3.6}\\ 
&\bullet F_{\leq h}(x) = \sum_{n=0}^{\infty} f_{\leq h}^nx^n =
  \frac{p_{h}(x)}{p_{h+1}(x)}\label{3.7} 
\end{align}
where
\begin{equation}
p_0(x)=1, \ \ \ p_1(x)=1-x, \ \ \  p_{h+1}(x)=p_h(x)-x\cdot p_{h-1}(x).\notag
\end{equation}
\end{corollary}
De Bruijn, Knuth, and Rice \cite{first} also found the explicit
formulas for the number of fixed-height ordered trees.  Based on their
results and Theorem \ref{main}, we have the following corollary: 

\begin{corollary}
$f_{\leq h}^1=f_{\leq h}^0 =1$.  For $n\geq 2$, $h \geq 0$, we have
\begin{align}
& f_{\leq h}^n=(h+3)^{-1}\sum_{1\leq j \leq \frac {h+2}{2}} 4^{n+1}
  \sin^2\left(\frac{j\pi}{h+3}\right)\cos^{2n}\left(\frac{j\pi}{h+3}
  \right).\label{general}  
\end{align}
\end{corollary}

\subsection{On labeled semiorders}
Let $g_h^n$ denote the number of nonisomorphic labeled semiorders with
$n$ elements and length $h$, and $g_{\leq h}^n$ denote the number of
nonisomorphic labeled semiorders with $n$ elements and length at most
$h$. Thus $g_h^n$ = $g_{\leq h}^n - g_{\leq (h-1)}^n$.  Let $G_h =
\sum_{n=0}^{\infty} g_h^n\frac{x^n}{n!}$ and $G_{\leq h} =
\sum_{n=0}^{\infty} g_{\leq h}^n\frac{x^n}{n!}.$ 

We obtain the exponential generating function $G_{h}$ from the
ordinary generating function $F_{h}$ by the following lemma, which is
due to Y. Zhang \cite{Yan}.  We first define equivalence of elements
and then state the lemma. 

\begin{definition}\label{5.5}
Two elements $p$ and $p'$ of a poset $P$ are \emph{equivalent} if 
$$ p'<q \Leftrightarrow p<q, \text{ for all } q\in P$$
and
$$ p'>q \Leftrightarrow p>q, \text{ for all } q\in P.$$
\end{definition}

\vspace{12pt}

\begin{lemma}\label{Yan}
Define the following two operations on an unlabeled poset $P$.
\begin{itemize}
\item The \emph{expansion} of $P$ at $p \in P$ is obtained from $P$ by
  adjoining a new element $p'$ such that $p$ and $p'$ are equivalent. 
\item The \emph{contraction} $c(P)$
of $P$ is a poset $c(P)$ obtained from $P$ by replacing every
equivalence class of elements with a single element. Call a poset $P$
a \emph{seed} if $P = c(P)$. Call a seed $P$  
\emph{rigid} if $P$ has no nontrivial automorphisms. 
\end{itemize}

Let $C$ be a family of unlabeled posets such that $C$ is closed under
expansion and contraction, and all seeds in $C$ are rigid. 
Let $F(x)=\sum_{P\in C}x^{\#P}$ and $G(x)=\sum_{P\in C}D_p
\frac{x^{\#P}}{(\#P)!}$, where $\#P$ is the number of elements in
poset $P$ and $D_p$ is the number of ways to label the elements of
$P$ up to isomorphism, i.e. $D_p=\frac{\#P}{\#(aut\hspace{3pt}P)}$, where $aut\hspace{3pt}P$ is the automorphism group of $P$. Then $G(x) = F(1-e^{-x})$. 
\end{lemma} 
The class of semiorders of length $h$ is closed under expansion and
contraction.  Zhang has observed that all $(\bm{2}+\bm{2})$-free seeds
are rigid.  Since semiorders are  $(\bm{2}+\bm{2})$-free, all seeds of
semiorders are rigid.  As a result, Lemma \ref{Yan} implies the
following corollary. 
\begin{corollary}
For $h\geq 0$,
\begin{equation}\label{3.8}
G_h(x)=F_h(1-e^{-x})=\frac{(1-e^{-x})^{h+1}}
{p_{h+1}(1-e^{-x})p_{h}(1-e^{-x})} 
\end{equation}
and
\begin{equation}\label{3.10}
G_{\leq h}(x)=F_{\leq h}(1-e^{-x})=F_{\leq h}(1-e^{-x}) =
\frac{p_{h}(1-e^{-x})}{p_{h+1}(1-e^{-x})}. 
\end{equation}
\end{corollary}

\section{Recurrence Relations}\label{sec5}
The generating functions and explicit formulas for the number of
semiorders of fixed length are complicated, but there are some concise
recurrence relations underneath.  We will discuss two useful
recurrence formulas in this section.  The first recurrence formula
(\ref{eq4.1}) is a standard result that is known for ordered trees
\cite[p.17]{first}, but only a proof using generating functions was
given, while the second recurrence formula (\ref{eq4.2}) is not
obvious for ordered trees.  We will provide concise combinatorial
proofs for both formulas.  In this way, we can better understand the
relations between fixed-length semiorders with different numbers of
elements.   

\subsection{Recurrence formula 1}
\begin{theorem}\label{5.7}
For $n \geq 2$ and $h \geq 1$, 
\begin{equation}\label{eq4.1}
f_{\leq h}^{n}=\sum_{t=0}^{n-1} f_{\leq h}^{t}f_{\leq h-1}^{n-1-t},
\end{equation}
where $f_{\leq h}^0=1.$
\end{theorem}

\begin{proof}
Let us prove this theorem by first defining the relative positions of elements on the same level of a semiorder.
\vspace{6pt}
\begin{definition}\label{rightmost}
For elements $a$ and $b$ on the same level of a semiorder $S$, we say that element $b$ is \emph{to the right of} element $a$ if $b$ is smaller than more elements than $a$ is, or $b$ and $a$ are smaller than the same number of elements while $b$ is larger than fewer elements than $a$ is.  
\end{definition}

\begin{rmk}
The above definition is unique up to isomorphism.  In fact, if semiorder $S$ has $n$ elements and say the
integer vector corresponding to semiorder $S$, as discussed in Section
\ref{ch1}, is $(r_1, r_2, \dots, r_n)$, then element $b$ is \emph{to the right
  of} element $a$ if $b$ corresponds to $r_j$, while $a$ corresponds
to $r_i$, $i<j$.  For detailed basic properties of semiorders, see \cite{linearextension}.
\end{rmk}

\vspace{6pt}
Let us now prove Theorem \ref{5.7}.  Let $a_1$ be the \emph{rightmost} element on the first level of $S$,
and let $T_1=\{a_1\}$.  Once $T_i$ is defined, let $T_{i+1}$ be the
set of elements on the $(i+1)^{\mathrm{th}}$ level, each of whose elements is
smaller than at least one element in $T_i$, $1 \leq i \leq h$.   For a
given semiorder $S$, the set $T_i$ is uniquely determined, $1 \leq i
\leq h$.  Notice that it is possible that $T_i=\emptyset$, for some
$i$, $1 \leq i \leq h$, and if $T_i=\emptyset$, then we must have
$T_j=\emptyset$ for all $i \leq j \leq h+1$.  Let $A_2 = T_1 \cup T_2
\cup \dots \cup T_{h+1}$, and $A_1 = A - A_2$, where $A$ is the set of
all elements of $S$.  Since $a_1 \in A_2$, we must have $1\leq |A_2|
\leq n$, $0\leq |A_1| \leq n-1$, and $|A_1|+|A_2|=n$. 

Let us separate $S$ into two semiorders $S_1$ and $S_2$.  Let $S_1$ be
the induced semiorder with element set $A_1$. Similarly, let $S_2$ be
the induced semiorder with element set $A_2$.  Let $S_3$ be the
semiorder obtained from $S_2$ by removing element $a_1$.  Then for a
given semiorder $S$, we have that $S_1$, $S_2$, $S_3$ are uniquely
defined.  Since $S$ is a semiorder of length at most $h$, semiorders
$S_1$ and $S_2$ have length at most $h$, and thus $S_3$ has length at
most $h-1$.   

Assume  $|A_1|=t$, so $S_3$ is a semiorder with $n-1-t$ elements.  As
a result, for a given $n$-element semiorder $S$ of length at most $h$,
we can uniquely obtain a pair of semiorders $S_1$ and $S_3$, of length
at most $h$ and at most $h-1$, and with $t$ and $n-1-t$ elements,
respectively, $0\leq t \leq n-1$. 

For example, if $S$ is as follows,
$$ 
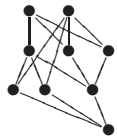
\xy 0;/r.25pc/: 
(0,0)*{\bullet}="b";
(5,0)*{\bullet}="c";
(0,-5)*{\bullet}="d";
(5,-5)*{\bullet}="e";
(10,-5)*{\bullet}="f";
(-2,-10)*{\bullet}="g";
(2,-10)*{\bullet}="h";
(8,-10)*{\bullet}="i";
(10,-15)*{\bullet}="j";
"b";"d"**\dir{-};
"b";"e"**\dir{-};
"b";"f"**\dir{-};
"c";"e"**\dir{-};
"c";"f"**\dir{-};
"c";"g"**\dir{-};
"c";"h"**\dir{-};
"d";"g"**\dir{-};
"d";"h"**\dir{-};
"d";"i"**\dir{-};
"e";"i"**\dir{-};
"f";"i"**\dir{-};
"g";"j"**\dir{-};
"h";"j"**\dir{-};
"i";"j"**\dir{-};
\endxy$$

\noindent the corresponding $S_1$, $S_2$, $S_3$ are:
$$ 
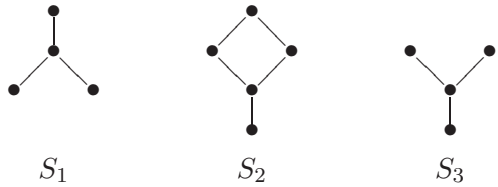
\xy 0;/r.25pc/: 
(-25,0)*{\bullet}="b1";
(-25,-5)*{\bullet}="c1";
(-30,-10)*{\bullet}="d1";
(-20,-10)*{\bullet}="e1";
(-25,-20)*{S_1};
"b1";"c1"**\dir{-};
"c1";"d1"**\dir{-};
"c1";"e1"**\dir{-};
(0,0)*{\bullet}="a";
(-5,-5)*{\bullet}="b";
(5,-5)*{\bullet}="c";
(0,-10)*{\bullet}="d";
(0,-15)*{\bullet}="e";
(0,-20)*{S_2};
"a";"b"**\dir{-};
"a";"c"**\dir{-};
"b";"d"**\dir{-};
"c";"d"**\dir{-};
"d";"e"**\dir{-};
(20,-5)*{\bullet}="b3";
(30,-5)*{\bullet}="c3";
(25,-10)*{\bullet}="d3";
(25,-15)*{\bullet}="e3";
(25,-20)*{S_3};
"b3";"d3"**\dir{-};
"c3";"d3"**\dir{-};
"d3";"e3"**\dir{-};
\endxy$$

On the other hand, given a pair of semiorders $S_1$ and $S_3$, of
length at most $h$ and at most $h-1$, and with $t$ and $n-1-t$
elements, respectively, $0\leq t \leq n-1$,  we can first add an
element $a_1$ to $S_3$, and let $a_1$ be larger than all other
elements in $S_3$.  Let us call the new semiorder $S_2$.  Then $S_2$
has $n-t$ elements and length at most $h$.   

Let us construct a new semiorder $S$ by combining $S_1$ and $S_2$ as
follows.  The elements on the $i^{\mathrm{th}}$ level of $S$ are the elements
on the $i^{\mathrm{th}}$ level of $S_1$ and $S_2$, $1\leq i \leq h+1$, and the
order relations in $S_1$ and $S_2$ are preserved.  In addition,  let
every element on the $(i-1)^{\mathrm{th}}$ level of $S_1$ be larger than all
elements on the $i^{\mathrm{th}}$ level of $S_2$.   

This construction uniquely gives an $n$-element semiorder $S$ of
length at most $h$, and if we separate $S$ into two semiorders by the
method discussed above, we get back $S_1$ and $S_3$. 

As a result, there is a one-to-one map between $n$-element semiorders
$S$ of length at most $h$, and pairs of semiorders $S_1$ and $S_3$, of
length $\leq h$ and $\leq h-1$, and with $t$ and $n-1-t$ elements,
respectively, $0\leq t \leq n-1$. 

Therefore, summing up possible $t$'s, we have
$$f_{\leq h}^{n}=\sum_{t=0}^{n-1} f_{\leq h}^{t}f_{\leq h-1}^{n-1-t}.$$
\end{proof}

\subsection{Recurrence formula 2}

\begin{theorem}\label{5.1}
For $n \geq 2$ and $h \geq 1$, we have 
\begin{equation}\label{eq4.2}
f_{\leq h}^n =
\sum_{k=1}^{\lfloor\frac{h+2}{2}\rfloor}(-1)^{k-1}\binom{h+2-k}{k}f_{\leq
  h}^{n-k}. 
\end{equation}
\end{theorem}
\vspace{6pt}

\begin{proof} Let us prove this theorem by first defining bad elements
  of a semiorder and then considering removing one or more bad
  elements from a given semiorder. 
\vspace{12pt}
\begin{definition}\label{badelement}
We say an element of a semiorder is \emph{bad} if the following two
conditions hold. 
\\$\bullet$ It is on the first level, or it is smaller than all
elements on the level immediately above it,  
\\and
\\$\bullet$ it is on the last level, or it is not larger than any
element on the level immediately below it. 
\end{definition}

\begin{rmk}\label{5.6}
By the above definition, there are no two adjacent levels which both
have bad elements.  In addition, by Proposition  \ref{2}, if two bad
elements are on the same level, they must be equivalent.  Therefore
there is at most one non-equivalent bad element on each level.  We
thus only consider one bad element on each level. 
\end{rmk}

\begin{proposition}\label{5.3}
For every semiorder, there exist bad elements.
\end{proposition}

\begin{proof}
Assume to the contrary that there is no bad element for some
semiorder.  Say $a$ is the rightmost element on the last level based on Definition \ref{rightmost}.  Since $a$ is not bad, there must be
some element $b$ on the last but one level which is not larger than
$a$.  Let $\{a_1, a_2, \dots, a_s\}$ be the set of all elements on the last but one level which are larger than $a$.  By the definition of element levels, we must have $s \geq 1$.  

If there exists some element $b_1$ on the last level such that $b>b_1$, since $a$ is the rightmost, i.e., $a$ is to the right of $b_1$, there must exist some $i$, $1\leq i \leq s$, such that $a_i$ is not larger than $b_1$.  Then $\{a_i>a, b>b_1\}$ forms a $(\bm{2}+\bm{2})$-structure.  This is a contradiction.  Hence $b$ is not larger than any element on the last level, and thus $b$ is the rightmost element on the last but one level.

Since $b$ is not bad, it
must not be on the first level, and there must be some element $c$ on
the level immediately above $b$ that is not larger than $b$.  Again due to the fact that $b$ is the rightmost element on its level and that we cannot have a $(\bm{2}+\bm{2})$-structure, $c$ must be not larger
than any element on the level immediately below the level $c$ is on.
Continuing, we can find an element on each level that is not larger
than any element on the level immediately below it.  Since the length
of the semiorder is finite, we can finally obtain an element $d$ on
the first level such that $d$ is not larger than any element on the
second level.  Then $d$ is bad, and we get a contradiction.  Hence,
for every semiorder, bad element exists. 
\end{proof}

We are now ready to deduce the recurrence formula (\ref{5.1}).  For
fixed $k$, consider adjoining $k$ bad elements to an $(n-k)$-element
semiorder to get a new semiorder.  Specifically, for an
$(n-k)$-element semiorder of length at most $h$, and for $k$
nonadjacent levels among levels $1, 2, \dots, h+1$,  we consider
adjoining $k$ elements onto the given levels of the semiorder in the
following manner.  Say we are adjoining an element onto the $l^{\mathrm{th}}$
level. 
\begin{itemize}
\item If $l=1$, let the new element be larger than all elements on the
  $i^{\mathrm{th}}$ level, $i \geq 3$, and be not comparable with any other
  element.   
\item If $l\geq2$, and the $(l-1)^{\mathrm{th}}$ level originally has at least
  one element, we let the new element be smaller than all elements on
  the $(l-1)^{\mathrm{th}}$ level, be larger than all elements on the $i^{\mathrm{th}}$
  level, $i \geq l+2$, and be not comparable with any other element.   
\item If $l\geq2$, and the $(l-1)^{\mathrm{th}}$ level originally has no
  element, we \emph{hang} the new element on the $l^{\mathrm{th}}$
  level, that is, we place it as an isolated vertex on the
  $l^{\mathrm{th}}$ level. We then call
  the new semiorder an \emph{invalid} semiorder, and if some semiorder $r_0$
  can be obtained from the invalid semiorder by taking out the hanging
  elements, we call the invalid semiorder the \emph{disguise} of $r_0$. 
\end{itemize}
 
By Definition \ref{badelement}, in the above adjoining, all new
elements which are not hung are bad elements in the new semiorder.
There are $\binom{h+2-k}{k}$ ways to choose $k$ nonadjacent levels
among levels $1, 2, \dots, h+1$.  Therefore, including multiplicity
and the invalid ones, we can obtain $\binom{h+2-k}{k}\cdot f_{\leq
  h}^{n-k}$ $n$-element semiorders from $(n-k)$-element semiorders by
adjoining $k$ elements.  Let $\mathcal{S}^k$ be the set of all such
$n$-element semiorders, including multiplicity.  Then
$|\mathcal{S}^k|=\binom{h+2-k}{k}\cdot f_{\leq h}^{n-k}$. 

Let $\mathcal{R}$ be the set of all $n$-element semiorders of length
at most $h$, and $\mathcal{R'}$ be the set of all semiorders with at
most $n-1$ elements and length at most $h$. 
By Proposition \ref{5.3}, every semiorder has bad elements, so every
semiorder $r\in \mathcal{R}$ can be obtained by the above process from
some $(n-k)$-element semiorder of length at most $h$ and $k$ given
nonadjacent levels, i.e., $r\in\mathcal{S}^k$, for some $1 \leq k \leq
\lfloor \frac{h+2}{2} \rfloor$.  However, $r$ might be in
$\mathcal{S}^k$ for multiple $k$'s, and $r$ may have multiple copies
in $\mathcal{S}^k$.  Meanwhile, $\mathcal{S}^k$ may contain some
semiorders not in $\mathcal{R}$, but are the disguises of some
semiorders $r' \in \mathcal{R'}$.  Notice that $|\mathcal{R}|=f_{\leq
  h}^n$.  In the following argument, we calculate the number of copies
of a semiorder in each $\mathcal{S}^k$ and obtain a formula connecting
$|\mathcal{R}|$ and $|\mathcal{S}^k|$, $1 \leq k \leq  \lfloor
\frac{h+2}{2} \rfloor$. 

For a semiorder $r_0\in \mathcal{R} \cup \mathcal{R'}$, let
$\mathcal{S}_{r_0}^k$ be the set of all semiorders in $\mathcal{S}^k$
which are equal to $r_0$ or a disguise of $r_0$.  Then $\mathcal{S}^k
= \bigcup_{r \in \mathcal{R} \cup \mathcal{R'}} \mathcal{S}_r^k$, and  
\begin{equation}\label{SR}
\sum_{r \in \mathcal{R} \cup \mathcal{R'}} |\mathcal{S}_r^k| =|\mathcal{S}^k|=\binom{h+2-k}{k}\cdot f_{\leq h}^{n-k}.
\end{equation}

Next, we show that
$\sum_{k=1}^{\lfloor\frac{h+2}{2}\rfloor}(-1)^{k-1}|\mathcal{S}_r^k|=1$,
for every $r \in \mathcal{R}$, and
$\sum_{k=1}^{\lfloor\frac{h+2}{2}\rfloor}(-1)^{k-1}|\mathcal{S}_{r'}^{k}|=0$,
for every $r'\in\mathcal{R'}$.

1. For a semiorder $r\in \mathcal{R}$, assume $r$ has $m$ bad
elements.  Since we adjoined $k$ elements to an $(n-k)$-element
semiorder to obtain the semiorder $r$, which has $n$ elements, the $k$
new elements should all be added to the levels among the $m$ levels
where the bad elements are, and no new element is hung.  So $k \leq
m$.  Further notice that for a given $k$, $1\leq k \leq m$, and given
$k$ levels among the $m$ levels where the bad elements are, there is a
unique $(n-k)$-element semiorder can be used to adjoin $k$ bad
elements to the chosen levels to obtain semiorder $r$.  There are
$\binom{m}{k}$ ways to choose the $k$ levels, so
$|\mathcal{S}_r^k|=\binom{m}{k}\cdot 1$, and  
\begin{equation} \label{R}
\sum_{k=1}^{\lfloor\frac{h+2}{2}\rfloor}(-1)^{k-1}|\mathcal{S}_r^k|=\sum_{k=1}^{m}(-1)^{k-1}|\mathcal{S}_r^k|=\sum_{k=1}^{m}(-1)^{k-1}\binom{m}{k}\cdot1=1.
\end{equation}

2. For a semiorder $r'\in \mathcal{R'}$, assume $r'$ has $n'$
elements, $m'$ of which are bad.  Further assume that semiorder $r'$
has length $h'$.  For a given $k$, $1\leq k \leq
\lfloor\frac{h+2}{2}\rfloor$, if we adjoined $k$ elements to an
$(n-k)$-element semiorder to obtain $r'$, we need to adjoin
$t'=n'-n+k$ elements to levels where the bad elements of $r'$ are, and
hang the remaining $n-n'$ elements.  Moreover, since we hung $n-n'$
elements, there should be at least $n-n'$ nonadjacent levels among
levels $h'+2, h'+3, \dots, h+1$.  As a result, $\left \lceil
\frac{h-h'}{2} \right \rceil \geq n-n'$. 

To obtain the semiorder $r'$, if we are given $t'$ levels among the
$m'$ levels where the bad elements are and $n-n'$ nonadjacent levels
among levels $h'+2, h'+3, \dots, h+1$, there is a unique
$(n-k)$-element semiorder to which we can adjoin $k$ elements to the
chosen levels to obtain the disguise of $r'$.  Notice that there are
$\binom{m'}{t'}$ ways to choose $t'$ levels among the $m'$ levels, and
$\binom{h-h'+1-(n-n')}{n-n'}$ ways to choose $n-n'$ nonadjacent levels
from levels $h'+2, h'+3, \dots, h+1$.  Thus,  
$$|\mathcal{S}_{r'}^k|=|S_{r'}^{t'-n'+n}|=
\binom{m'}{t'}\cdot\binom{h-h'+1-(n-n')}{n-n'}\cdot
1.$$
 
Then
\begin{align}\label{Rr}
\sum_{k=1}^{\lfloor\frac{h+2}{2}\rfloor}(-1)^{k-1}|\mathcal{S}_{r'}^k|
&=\sum_{t'=0}^{m'}(-1)^{t'-n'+n-1}|\mathcal{S}_{r'}^{t'-n'+n}|\notag\\
&=\sum_{t'=0}^{m'}(-1)^{t'-n'+n-1}\binom{m'}{t'}\cdot
\binom{h-h'+1-(n-n')}{n-n'}\cdot 1\notag\\ 
&=(-1) ^{n-n'-1}\cdot \binom{h-h'+1-(n-n')}{n-n'}
\sum_{t'=0}^{m'}(-1)^{t'}\binom{m'}{t'}=0. 
\end{align}

However, we should be careful with the special case when $t' \geq 1$
and the $(h'+1)^{\mathrm{th}}$ level of $r'$ has bad elements.  When we choose
$t'$ levels among the $m'$ levels where bad elements are and $n-n'$
nonadjacent levels among levels $h'+2, h'+3, \dots, h+1$, it is
possible that the $(h'+1)^{\mathrm{th}}$ and $(h'+2)^{\mathrm{th}}$ levels are both
chosen.  This case should not occur when we directly choose $k$
nonadjacent levels among levels $1, 2, \dots, h+1$.  So we need to
take out the overcounts, and thus in this case,  
\begin{align*}
|\mathcal{S}_{r'}^k|&=|\mathcal{S}_{r'}^{t'-n'+n}|\\
&=\binom{m'}{t'}\cdot\binom{h-h'+1-(n-n')}{n-n'}\cdot 1
-\binom{m'-1}{t'-1}\cdot\binom{h-h'-1-(n-n'-1)}{n-n'-1}\cdot 1 
\end{align*}

By similar calculations, we have
$\sum_{k=1}^{\lfloor\frac{h+2}{2}\rfloor}(-1)^{k-1}|\mathcal{S}_{r'}^k|=0$.

To conclude the proof, by equations (\ref{SR}), (\ref{R}), and
(\ref{Rr}), we have 
\begin{align*}
f_{\leq h}^{n} = |\mathcal{R}| = \sum _{r\in \mathcal{R}} 1 +
\sum_{r'\in \mathcal{R'}} 0 
&= \sum_{r\in \mathcal{R}}
\sum_{k=1}^{\lfloor\frac{h+2}{2}\rfloor}(-1)^{k-1}|\mathcal{S}_r^k|+\sum_{r'\in
  \mathcal{R'}}
\sum_{k=1}^{\lfloor\frac{h+2}{2}\rfloor}(-1)^{k-1}|\mathcal{S}_{r'}^k|\\ 
&=\sum_{k=1}^{\lfloor\frac{h+2}{2}\rfloor}(-1)^{k-1}\sum_{r\in
  \mathcal{R}\cup \mathcal{R'}}
|\mathcal{S}_{r}^k|\\&=\sum_{k=1}^{\lfloor\frac{h+2}{2}\rfloor}(-1)^{k-1}\binom{h+2-k}{k}\cdot
f_{\leq h}^{n-k}. 
\end{align*}
\end{proof}

\section{The Number of Semiorders of Small Length}\label{sec6}
We can substitute certain lengths $H$ in the explicit formulas for the
number of semiorders.  Though the original formulas are very
complicated, we can get some simple results for small values of $H$.
In this section, we list these simple results and give bijective
proofs, which present a clearer view of the number of fixed-length
semiorders. 

\subsection{$f_{\leq 1}^n$, the number of nonisomorphic unlabeled
  $n$-element semiorders of length at most one} 
\begin{theorem}
For $n\geq 1$, $f_{\leq 1}^n=2^{n-1}$.  
\end{theorem}
We give a simple bijective proof here. 
\begin{proposition} 
For $n$ elements $a_1, a_2, \dots, a_n$, put $a_1$ on the upper level,
and each of $a_2, \dots, a_n$ either on the lower or upper level.
Define the order relations in the following way: 
$a_i >a_j$ if and only if $i<j$, and $a_i$ is on the upper level while
$a_j$ is on the lower level. 
\vspace{6pt}

We claim that the above defines a bijective map from (a) an
arrangement of $n-1$ elements onto two levels in (b) an $n$-element
semiorder of length at most one. 
\end{proposition}

Here is an example of the map.  Say $n=10$, and for $a_2, \dots,
a_{10}$, let $\{a_2, a_5, a_9\}$ be on the upper level, and $\{a_3,
a_4, a_6, a_7, a_8, a_{10}\}$ on the lower level.  Then the
corresponding semiorder looks like: 

$$ 
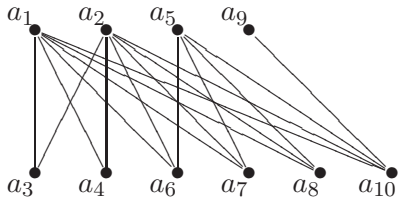
\xy 0;/r.45pc/: 
(0,0)*{\bullet}="a";
(5,0)*{\bullet}="b";
(10,0)*{\bullet}="c";
(15,0)*{\bullet}="d";
(0,-10)*{\bullet}="f";
(5,-10)*{\bullet}="g";
(10,-10)*{\bullet}="h";
(15,-10)*{\bullet}="i";
(20,-10)*{\bullet}="j";
(25,-10)*{\bullet}="k";
(-1,1)*{a_1};
(4,1)*{a_2};
(9,1)*{a_5};
(14,1)*{a_9};
(-1,-11)*{a_3};
(4,-11)*{a_4};
(9,-11)*{a_6};
(14,-11)*{a_7};
(19,-11)*{a_{8}};
(24,-11)*{a_{10}};
"a";"f"**\dir{-};
"a";"g"**\dir{-};
"a";"h"**\dir{-};
"a";"i"**\dir{-};
"a";"j"**\dir{-};
"a";"k"**\dir{-};
"b";"f"**\dir{-};
"b";"g"**\dir{-};
"b";"h"**\dir{-};
"b";"i"**\dir{-};
"b";"j"**\dir{-};
"b";"k"**\dir{-};
"c";"h"**\dir{-};
"c";"i"**\dir{-};
"c";"j"**\dir{-};
"c";"k"**\dir{-};
"d";"k"**\dir{-};
\endxy$$

\begin{proof}
We first show that the map gives a semiorder of length at most one.
It suffices to show that the poset the map gives is indeed a
semiorder.  Then the only possible violation is a
$(\bm{2}+\bm{2})$-structure.  If there exist four distinct elements
$a_i$, $a_j$, $a_k$, $a_m$, such that $a_i>a_j$, $a_k>a_m$, $a_i\sim
a_k$, $a_i\sim a_m$, $a_k\sim a_j$, $a_m\sim a_j$, then $a_i, a_k$
must be on the upper level, while $a_j, a_m$ must be on the lower
level, and $i<j$, $k<m$.  Since $a_i\sim a_m$, we must have $i>m$;
since $a_k\sim a_j$, we must have $k>j$.  Then $k>j>i>m>k$, which is a
contradiction.  Hence the map gives a semiorder of length at most
one. 
  
We then claim that the inverse map is also well-defined, and thus the
map is bijective.  For a given $n$-element semiorder $r$ with $m$
elements on the upper level, let $\rho(r)=(r_1, r_2, \dots, r_n)$.
Then $r_{m+1} = r_{m+2} = \dots = r_n=0$.  Say element $a_{t_i}$
corresponds to $r_i$, $1 \leq i\leq m$, and then there should be
exactly $r_i$ elements on the lower level such that their subscripts
are larger than $t_i$.  As a result, note that $a_1$ is on the upper
level, we should also have $a_{r_1-r_2+2}, a_{r_1-r_3+3}, \dots,
a_{r_1-r_m+m}$ on the upper level.  In other words, for a given
semiorder of length at most one, the elements arranged on the upper
level are uniquely determined.  Therefore, the inverse map is
well-defined. 

\end{proof}

As an example, if we have $(r_1, r_2, \dots, r_n)=(6, 6, 4, 1, 0, 0,
0, 0, 0, 0)$, then the elements on the upper level must be $a_1, a_2,
a_5, a_9$.

There are $2^{n-1}$ ways to arrange elements $a_2, \dots, a_n$ on
either upper or lower level, and thus there are $2^{n-1}$
nonisomorphic unlabeled $n$-element semiorders of length at most one.

\subsection{The number of nonisomorphic trees derived from semiorders of length at most one}\label{tree}

In this subsection, we take a closer look at the unlabeled semiorders
of length at most one. For an $n$-element semiorder $S$ of length at
most one, and exactly $m$ elements on the first level, let
$\rho(S)=(r_1, r_2, \dots, r_m, 0, 0, \dots, 0)$, where there are
$n-m$ $0$'s and $n-m\geq r_1 \geq r_2 \geq \dots \geq r_m \geq 0$.
Let the elements of the semiorder be $s_1, s_2, \dots, s_n$, with
$s_i$ corresponding to $r_i$, $1\leq i \leq m$, and then $s_1, s_2,
\dots, s_m$ are on the upper level. 

For a permutation $\sigma=(a_1, a_2, \dots, a_m)$ of $\{1, 2, \dots,
m\}$, if we add the relations $s_{a_1}>s_{a_2}>\dots>s_{a_m}$ to the
original semiorder, we get a tree with the main trunk
$s_{a_1}>s_{a_2}>\dots>s_{a_m}$, and the elements $s_{m+1}, s_{m+2},
\dots, s_n$ attached to one of the elements on the main trunk in the
following manner.  For $m+1\leq i \leq n$, element $s_i$ is attached
to element $s_j$ on the main trunk if and only if $s_i < s_j$, and
$s_i$ is incomparable with all elements on the main trunk that are
below $s_j$, $1\leq j \leq m$.  We denote the tree derived from
semiorder $S$ and permutation $\sigma$ by $T(S, \sigma)$. 

For example, the Hasse diagram of the semiorder $S$ with $\rho(S)=(7,
5, 4, 2, 1, 0, 0, 0, 0, 0, 0)$ is as follows: 
$$ 
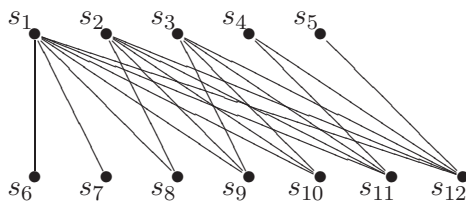
\xy 0;/r.45pc/: 
(0,0)*{\bullet}="a";
(5,0)*{\bullet}="b";
(10,0)*{\bullet}="c";
(15,0)*{\bullet}="d";
(20,0)*{\bullet}="e";
(0,-10)*{\bullet}="f";
(5,-10)*{\bullet}="g";
(10,-10)*{\bullet}="h";
(15,-10)*{\bullet}="i";
(20,-10)*{\bullet}="j";
(25,-10)*{\bullet}="k";
(30,-10)*{\bullet}="l";
(-1,1)*{s_1};
(4,1)*{s_2};
(9,1)*{s_3};
(14,1)*{s_4};
(19,1)*{s_5};
(-1,-11)*{s_6};
(4,-11)*{s_7};
(9,-11)*{s_8};
(14,-11)*{s_9};
(19,-11)*{s_{10}};
(24,-11)*{s_{11}};
(29,-11)*{s_{12}};
"a";"f"**\dir{-};
"a";"g"**\dir{-};
"a";"h"**\dir{-};
"a";"i"**\dir{-};
"a";"j"**\dir{-};
"a";"k"**\dir{-};
"a";"l"**\dir{-};
"b";"h"**\dir{-};
"b";"i"**\dir{-};
"b";"j"**\dir{-};
"b";"k"**\dir{-};
"b";"l"**\dir{-};
"c";"i"**\dir{-};
"c";"j"**\dir{-};
"c";"k"**\dir{-};
"c";"l"**\dir{-};
"d";"k"**\dir{-};
"d";"l"**\dir{-};
"e";"l"**\dir{-};
\endxy$$

\vspace{6pt}

Suppose that $\sigma=(1, 5, 3, 2, 4)$, and then we add the relations
$s_1>s_5>s_3>s_2>s_4$ to the original semiorder.  Then $T(S,\sigma)$
is 
$$ 
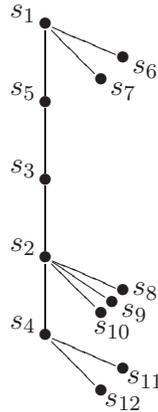
\xy 0;/r.35pc/: 
(0,0)*{\bullet}="a";
(0,-21)*{\bullet}="b";
(0,-14)*{\bullet}="c";
(0,-28)*{\bullet}="d";
(0,-7)*{\bullet}="e";
(7,-3)*{\bullet}="f";
(5,-5)*{\bullet}="g";
(7,-24)*{\bullet}="h";
(6,-25)*{\bullet}="i";
(5,-26)*{\bullet}="j";
(7,-31)*{\bullet}="k";
(5,-33)*{\bullet}="l";
(-2,1)*{s_1};
(-2,-20)*{s_2};
(-2,-13)*{s_3};
(-2,-27)*{s_4};
(-2,-6)*{s_5};
(9,-4)*{s_6};
(7,-6)*{s_7};
(9,-24)*{s_8};
(8,-26)*{s_9};
(6,-27.5)*{s_{10}};
(9,-32)*{s_{11}};
(7,-34)*{s_{12}};
"a";"e"**\dir{-};
"a";"f"**\dir{-};
"a";"g"**\dir{-};
"e";"c"**\dir{-};
"c";"b"**\dir{-};
"b";"d"**\dir{-};
"b";"h"**\dir{-};
"b";"i"**\dir{-};
"b";"j"**\dir{-};
"d";"k"**\dir{-};
"d";"l"**\dir{-};
\endxy$$

Here we call $s_1>s_5>s_3>s_2>s_4$ the main trunk, and say elements
$s_6, s_7$ are attached to $s_1$, elements $s_8, s_9, s_{10}$ are
attached to $s_2$, and elements $s_{11}, s_{12}$ are attached to
$s_4$.   

The idea of transforming a semiorder of length at most one to a tree
is suggested by R. Stanley, in the context of finding the number of
linear extensions of $n$-element semiorders of length at most one.
Though this idea may not be useful in its original context, we can
give a different application. 

\begin{theorem}\label{catalan}
Given an $n$-element semiorder $R_{m}$ of length at most one and
exactly $m$ elements on the first level, let $\rho(R_{m})=(r_1, r_2,
\dots, r_n)$.  If $r_i \neq r_j$ for any $1 \leq i < j \leq m$,  then
the number of nonisomorphic unlabeled trees in $\{T(R_{m}, \sigma) |
\sigma \in S_m\}$ is the Catalan number $C_m$. 
\end{theorem}

\begin{proof}
A permutation $\sigma=(a_1, a_2, \dots, a_m)$ of $\{1, 2, \dots, m\}$
uniquely determines the main trunk.  For $m+1\leq i\leq n$, assume
element $s_i$ is smaller than $t_i$ elements.  Then $s_i$ must be
smaller than $s_1, s_2, \dots, s_{t_i}$ and is not comparable with the
other elements.  Let $s_{u_i}$ be the lowest element on the main trunk
among $s_1, s_2, \dots, s_{t_i}$.  Then $s_i$ is attached to $s_{u_i}$
as a leaf, meaning $s_i$ is smaller than $s_{u_i}$ but not comparable
with any element  on the main trunk below $s_{u_i}$.   

As a result, for an element $s_{u}$ on the main trunk, $s_u$ has
leaves only if it is the lowest element on the main trunk among $s_1,
s_2, \dots, s_t$, for some $1\leq t \leq m$.  In other words, assume
$b_1 < b_2< \dots < b_k$ are the set of right-to-left minima of the
permutation $\sigma=(a_1, a_2, \dots, a_m)$, and then only the
elements $s_{b_1},  s_{b_2}, \dots,  s_{b_k}$ may have leaves.
Further notice that the numbers of leaves attached to $s_{b_1},
s_{b_2}, \dots,  s_{b_k}$ are $r_{b_1}-r_{b_2}, r_{b_2}-r_{b_3},
\dots, r_{b_{k-1}}-r_{b_k}, r_{b_k}$, respectively, and $r_i \neq r_j$
for any $1 \leq i < j \leq m$.  Therefore, for a given $n$-element
semiorder $R_{m}$ of length at most one and exactly $m$ elements on
the first level, the value and position of the right-to-left minima of
the permutation $\sigma=(a_1, a_2, \dots, a_m)$ uniquely determines
$T(R_m,\sigma)$.   

For instance, in the example above, we have $\rho(R_5)=(r_1, r_2,
\dots, r_{12})=(7, 5, 4, 2, 1, 0, 0, 0, 0, 0, 0, 0)$ and $\sigma=(1,
5, 3, 2, 4)$.  The right-to-left-minima of $\sigma$ and their
positions with $\sigma$ is given by $1, \ast, \ast, 2, 4$.  Then the
corresponding tree $T(R_5, \sigma)$ has five nodes on the main trunk,
$r_1-r_2=2$ leaves attached to the first node, $r_2-r_4=3$ leaves
attached to the forth node, and $r_4=2$ leaves attached to the fifth
node. 

Therefore, the number of all possible nonisomorphic unlabeled trees in
$\{T(R_{m}, \sigma) | \sigma \in S_m\}$ is the number of ways to
specify the values and positions of the right-to-left minima of
permutations $\sigma \in S_m$. That is, if we let
$RtLM(\sigma)=\{(a,\sigma(a)) | 1\leq a \leq m, \sigma(a) \text{ is a
  right-to-left minima in }\sigma\}$, then $\#\{T(R_{m}, \sigma) |
\sigma \in S_m\}=\#\{RtLM(\sigma) | \sigma \in S_m\}$.  We calculate
$\#\{RtLM(\sigma) | \sigma \in S_m\}$ in the following lemma. 

\begin{lemma}\label{Narayana}
Let $RL(\sigma)$ be the number of right-to-left minima of the
permutation $\sigma$.  For $1\leq k \leq m$, let
$f(m,k)=\#\{RtLM(\sigma) | \sigma \in S_m, RL(\sigma)=k\}$.  Then
$f(m,k)=N(m,k)=\frac{1}{m}\binom{m}{k}\binom{m}{k-1}$, a Narayana
number. 
\end{lemma}

For example, for $m=3$ and $k=2$, $f(3,2)=\#\{RtLM(\sigma) | \sigma
\in S_3, RL(\sigma)=2\}=\#\{RtLM(\sigma) | \sigma=(23), (12), \text{or
}(132)\}=\#\{\{(3,2), (1,1)\}, \{(3,3), (2,1)\}, \{(3,2),
(2,1)\}\}=3$. 

\begin{proof}
For $1 \leq k \leq m$, the Narayana number $N(m,k)$ is equal to the
number of Dyck paths of semilength $m$ with $k$ peaks, which are the
turning points from a $(1,1)$ step to a $(1,-1)$ step on the path.  We
prove the lemma by establishing a bijection between (i) Dyck paths of
semilength $m$ with $k$ peaks, and (ii) the collection of different
$RtLM(\sigma)$'s for $\sigma \in S_m, RL(\sigma)=k$.  We define a map
from (i) to (ii) as follows: 
\vspace{6pt}

Given a Dyck path of semilength $m$ with $k$ peaks, let us read the
Dyck path from left to right and do the following: 
\begin{itemize}
\item Label the endpoints of $(1,1)$ steps from left to right with $1$
  to $m$. Since there are $m$ $(1,1)$ steps, there should be $m$ such
  endpoints.  
\item Label the startpoints of $(1,-1)$ steps from left to right with
  $1$ to $m$. Since there are $m$ $(1,-1)$ steps, there should be $m$
  such startpoints.  
\item Notice that a point on the Dyck path is a peak if and only if it
  is both an endpoint of a $(1,1)$ step and a startpoint of a $(1,-1)$
  step.  Let $(i,j)$ be the coordinate of a peak, if the peak is the
  $i^{\mathrm{th}}$ endpoint and the $j^{\mathrm{th}}$ startpoint.  Assume the
  coordinate of the $k$ peaks are $(a_1, b_1), (a_2, b_2), \dots,
  (a_k, b_k)$. Then $a_1<a_2<\dots<a_k=m$ and  $1=b_1<b_2<\dots<b_k$. 
\item Obtain a specification of the values and positions of the
  right-to-left minima of a permutation by putting the number $b_i$ on
  position $a_i$, $1\leq i \leq k$. 
\end{itemize}

\begin{proposition}\label{b2}
The above map is valid, i.e., $\{(a_i, b_i), 1\leq i \leq
k\}=RtLM(\sigma)$, for some $\sigma \in S_m, RL(\sigma)=k$, and the
above map is a bijection. 
\end{proposition}

For example, if we have the following Dyck path:
$$ 
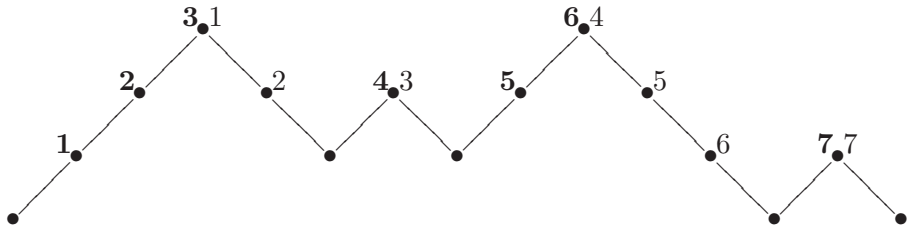
\xy 0;/r.40pc/: 
(0,0)*{\bullet}="a";
(5,5)*{\bullet}="b";
(10,10)*{\bullet}="c";
(15,15)*{\bullet}="d";
(20,10)*{\bullet}="e";
(25,5)*{\bullet}="f";
(30,10)*{\bullet}="g";
(35,5)*{\bullet}="h";
(40,10)*{\bullet}="i";
(45,15)*{\bullet}="j";
(50,10)*{\bullet}="k";
(55,5)*{\bullet}="l";
(60,0)*{\bullet}="m";
(65,5)*{\bullet}="n";
(70,0)*{\bullet}="o";
(4,6)*{\bf{1}};
(9,11)*{\bf{2}};
(14,16)*{\bf{3}};
(29,11)*{\bf{4}};
(39,11)*{\bf{5}};
(44,16)*{\bf{6}};
(64,6)*{\bf{7}};
(16,16)*{1};
(21,11)*{2};
(31,11)*{3};
(46,16)*{4};
(51,11)*{5};
(56,6)*{6};
(66,6)*{7};
"a";"b"**\dir{-};
"b";"c"**\dir{-};
"c";"d"**\dir{-};
"d";"e"**\dir{-};
"e";"f"**\dir{-};
"f";"g"**\dir{-};
"g";"h"**\dir{-};
"h";"i"**\dir{-};
"i";"j"**\dir{-};
"j";"k"**\dir{-};
"k";"l"**\dir{-};
"l";"m"**\dir{-};
"m";"n"**\dir{-};
"n";"o"**\dir{-};
\endxy$$

The semilength of the Dyck path is $m=7$, and it has $4$ peaks.  The
numbers in bold face are the labels for the endpoints of $(1, 1)$
steps, and the numbers in ordinary type are the labels for the
startpoints of $(-1, 1)$ steps.  Then the coordinates of the 4 peaks
are $(3, 1), (4, 3), (6, 4), (7,7)$. We put 1 on position 3, 3 on
position 4, 4 on position 6, 7 on position 7, and then we get a
possible specification of the values and positions of the
right-to-left minima of a permutation: 
\begin{equation}\label{permutation}
\ast, \ast, 1, 3, \ast, 4, 7.
\end{equation}

\begin{proof}
\begin{itemize}
\item We first show that the above map gives us a valid specification
  of the values and positions of the right-to-left minima of some
  permutation $\sigma \in S_m$.  We prove the validity by constructing
  such $\sigma$. 

Assume the labels of the $m-k$ endpoints of the $(1,1)$ steps which
are not peaks, are $c_1< c_2 < \dots < c_{m-k}$; assume the labels of
the $m-k$ startpoints of the $(1,-1)$ steps which are not peaks are
$d_1< d_2 < \dots < d_{m-k}$.  Then 
$$\{a_1, a_2, \dots, a_k, c_1, c_2,
\dots, c_{m-k}\}=\{b_1, b_2, \dots, b_k, d_1, d_2, \dots,
d_{m-k}\}=\{1, 2, \dots, m\}. $$
Since the path never goes below the
$x$-axis, we must have $c_j < d_j$ for every $1 \leq j \leq m-k$.  

Let $\sigma$ be a permutation such that $\sigma(a_i)=b_i$, $1 \leq i
\leq k$, and $\sigma(c_j)=d_j$, $1 \leq j \leq m-k$.  We will show
that $RtLM(\sigma)=\{(a_i, b_i), 1\leq i \leq k\}$. 

For every $j$, $1 \leq j \leq m-k$, since the Dyck path never goes
below the $x$ axis, there must be a peak $i$ on the path between $c_j$
and $d_j$.  Then $c_j<a_i$ and $b_i<d_j$, and thus in $\sigma$,
$b_i=\sigma(a_i)$ is smaller than $d_j=\sigma(c_j)$, while $a_i>c_j$,
i.e., $b_i$ is to the right of $d_j$.  Therefore $d_j$ cannot be a
right-to-left minimum. 

On the other hand, for every $i$, $1 \leq i \leq k$, there does not
exist some $1\leq i'\leq k$ such that $b_{i'}<b_i$ and $a_{i'}>a_i$.
In addition, if there exists some $1\leq j \leq m-k$ such that
$d_{j}=\sigma(c_j)<b_i=\sigma(a_i)$ and $c_{j}<a_i$, then by the above
paragraph there exists $1\leq i'\leq k$, such that $b_{i'}<d_j<b_i$
and $a_{i'}>c_j>a_i$.  We obtain a contradiction.  As a result,
$RtLM(\sigma)=\{(a_i, b_i), 1\leq i \leq k\}$. 

For example, let us construct the permutation $\sigma$ for the above
example: $c_1=1$, $c_2=2$, $c_3=5$ and $d_1=2$, $d_2=5$, $d_3=6$.  We
obtain $\sigma=(2, 5, 1, 3, 6, 4, 7)$, and this permutation exactly
corresponds to the right-to-left minima as shown in Example
\ref{permutation}. 

\item We now show that the inverse of the map is well-defined, and
  thus the map is bijective.   

For a given specification $\{(a_i, b_i), 1\leq i \leq
k\}=RtLM(\sigma)$, for some $\sigma \in S_m$, we have
$a_1<a_2<\dots<a_k=m$ and $1=b_1<b_2<\dots<b_k$.  We construct the
corresponding Dyck path as follows:  when we walk along the path from
left to right, we first walk up $a_1$ steps and then turn down, and
walk down $b_2 -b_1$ steps and then turn up.  We continue to walk up
$a_2-a_1$ steps and then turn down, and walk down $b_3 -b_2$ steps and
then turn up.  In general, we walk up $a_i-a_{i-1}$ steps and then
turn down, and walk down $b_{i+1}-b_{i}$ steps, $2\leq i \leq m-1$.
In the end, we walk up $a_m-a_{m-1}$ steps and walk down $m+1-b_m$
steps.  During the walk, we walk up in total
$a_1+(a_2-a_1)+\dots+a_m-a_{m-1}=a_m=m$ steps, and walk down in total
$(b_2 -b_1)+\dots+(b_m-b_{m-1})+m+1-b_m = m+1-b_1=m$ steps, and we
make $k$ turns from up to down.   Since $\{(a_i, b_i), 1\leq i \leq
k\}$ is a collection of right-to-left minima of some permutation, for
any $1\leq i \leq k-1$, $a_i\leq a_{i+1}-1 \leq
b_{i+1}-1=b_{i+1}-b_1$, so we never walk below the $x$-axis on the
path. Therefore we get a unique Dyck path with semilength $m$ and $k$
peaks.  Hence the inverse map is well-defined, so Lemma \ref{Narayana}
is proved. 
\end{itemize}
\end{proof}
\end{proof}

To conclude the proof of Theorem \ref{catalan}, note
that the Catalan number $C_m = \sum_{k=1}^{m}N(m,k)$.  Hence, we have that the number of $\text{ nonisomorphic unlabeled trees in
    $\{T(R_{m}, \sigma) | \sigma \in S_m\}$}$ is equal to
$\#\{RtLM(\sigma) | \sigma \in S_m\}=\sum_{k=1}^{m}f(m,k)=\sum_{k=1}^{m}N(m,k)=C_m=\frac{1}{m+1}\binom{2m}{m}$.
\end{proof}

\begin{rmk}
Theorem \ref{catalan}, along with Lemma \ref{Narayana}, gives
another combinatorial explanation of the Catalan number. 
\end{rmk}

\subsection{The generating function for the number of nonisomorphic
  labeled $n$-element semiorders of length at most one} 

Recall that an ordered partition of a set is a partition of the set
into some pairwise disjoint nonempty subsets, together with a linear
ordering of these subsets.  From the generating function (\ref{3.8})
for $G_{\leq h}(x)$, we get $G_{\leq
  1}(x)=(1-e^{-x})/(2e^{-x}-1)=(e^{x}-1)/(2-e^{x})$, which is
exactly the exponential generating function for the number of ordered
partitions \cite[p.472]{Stanley}.  As a result, we can get the
following theorem: 

\vspace{6pt}

\begin{theorem}\label{6.3}
The number of nonisomorphic labeled $n$-element semiorders of length
at most one is equal to the number of ordered partitions of $[n]$. 
\end{theorem}
We give a simple bijective proof to Theorem \ref{6.3}. 
\vspace{6pt}

\begin{proposition}\label{b3}
For an ordered partition $(A_1, \dots, A_k)$ of $[n]$, let
$|A_i|=a_i$,  $1\leq i \leq k$, so $n=a_1+a_2+\dots +a_k$.  Define the
semiorder $R$ by $\rho(R)=(m, m-1, \dots, 1, 0, \dots, 0)$, where
$m=\lfloor\frac{k}{2}\rfloor$, and there are $\lceil\frac{k}{2}\rceil$
$0$'s.  Then $R$ has $k$ elements. Say the elements are $t_1, t_2,
\dots, t_k$, with $t_i$ corresponding to the $i^{\mathrm{th}}$ entry of $R$'s
integer vector.  Let $R'$ be another semiorder such that its
contraction $c(R')$ (defined in Lemma \ref{Yan}) is $R$, and in $R'$,
the sizes of the equivalence classes are $a_1, a_2, \dots, a_k$,
respectively, with $a_i$ corresponding to equivalence class
$t_i$. Label the elements in the $i^{\mathrm{th}}$ equivalence class with the
corresponding numbers in $A_i$, $1\leq i \leq k$. 

We claim that the above defines a bijective map from ordered
partitions of $[n]$ to $n$-element labeled semiorders $R'$ of length
at most one. 

\end{proposition}

For example, if we have an ordered partition
$\{1,4\}\{2,6,8\}\{7\}\{3,5\}$, then $k=4$, $m=2$, and $(a_1, a_2,
a_3, a_4)=(2, 3, 1, 2)$.  We have $\rho(R)=(2, 1, 0, 0)$, and then the
map works as follows: 

$$ 
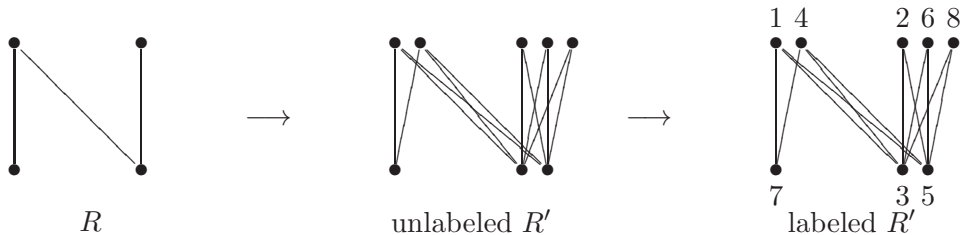
\xy 0;/r.80pc/: 
(-15,0)*{\bullet}="a0";
(-15,-5)*{\bullet}="b0";
(-10,0)*{\bullet}="c0";
(-10,-5)*{\bullet}="d0";
(-12,-7)*{R};
(-5,-3)*{\longrightarrow};
"a0";"b0"**\dir{-};
"a0";"d0"**\dir{-};
"c0";"d0"**\dir{-};
(0,0)*{\bullet}="$a1$";
(1,0)*{\bullet}="$a2$";
(0,-5)*{\bullet}="b";
(5,0)*{\bullet}="$c1$";
(6,0)*{\bullet}="$c2$";
(7,0)*{\bullet}="$c3$";
(5,-5)*{\bullet}="$d1$";
(6,-5)*{\bullet}="$d2$";
(3,-7)*{\text{unlabeled } R'};
(10,-3)*{\longrightarrow};
"$a1$";"b"**\dir{-};
"$a2$";"b"**\dir{-};
"$c1$";"$d1$"**\dir{-};
"$c2$";"$d1$"**\dir{-};
"$c3$";"$d1$"**\dir{-};
"$c1$";"$d2$"**\dir{-};
"$c2$";"$d2$"**\dir{-};
"$c3$";"$d2$"**\dir{-};
"$a1$";"$d1$"**\dir{-};
"$a2$";"$d1$"**\dir{-};
"$a1$";"$d2$"**\dir{-};
"$a2$";"$d2$"**\dir{-};
(15,0)*{\bullet}="$a12$";
(16,0)*{\bullet}="$a22$";
(15,-5)*{\bullet}="b2";
(20,0)*{\bullet}="$c12$";
(21,0)*{\bullet}="$c22$";
(22,0)*{\bullet}="$c32$";
(20,-5)*{\bullet}="$d12$";
(21,-5)*{\bullet}="$d22$";
(18,-7)*{\text{labeled } R'};
(15,1)*{1};
(16,1)*{4};
(15,-6)*{7};
(20,1)*{2};
(21,1)*{6};
(22,1)*{8};
(20,-6)*{3};
(21,-6)*{5};
"$a12$";"b2"**\dir{-};
"$a22$";"b2"**\dir{-};
"$c12$";"$d12$"**\dir{-};
"$c22$";"$d12$"**\dir{-};
"$c32$";"$d12$"**\dir{-};
"$c12$";"$d22$"**\dir{-};
"$c22$";"$d22$"**\dir{-};
"$c32$";"$d22$"**\dir{-};
"$a12$";"$d12$"**\dir{-};
"$a22$";"$d12$"**\dir{-};
"$a12$";"$d22$"**\dir{-};
"$a22$";"$d22$"**\dir{-};
\endxy$$

\begin{proof}
We first show that the map takes an ordered partition of $[n]$ to a labeled $n$-element semiorder of length at most one.
The size of every equivalence class of $R$ is one, and thus semiorder
$R$ is a valid contraction.  The length of $R$ is at most one, and
thus $R'$ also has length at most one.  In addition, $R$ has $k$
elements, so $R'$ has $k$ equivalence classes.  Therefore, we can
construct the equivalence classes of $R$ to have sizes $a_1, a_2,
\dots, a_k$.  Moreover, within an equivalence class with $a_i$
elements, $1\leq i \leq k$, since we only consider nonisomorphic
semiorders, it does not matter which of the $a_i$ numbers in $A_i$ is
assigned to which element in this equivalence class.  Hence the way to
label elements is unique up to isomorphism.  Thus each ordered
partition of $[n]$ uniquely corresponds to a labeled $n$-element
semiorder of length at most one. 

Next we show that the inverse map is well-defined and uniquely
determines an ordered partition of $[n]$.  Let the labeled $n$-element
semiorder $R'$ of length at most one have $k$ equivalence classes.
Let us group up the labels within every equivalence class, so 
$$[n] = \bigcup \{\text{labels in each equivalence class}\}. $$  

To obtain an ordered partition of $[n]$, it suffices to find the way
to order the $k$ equivalence classes of $R'$. The contraction $c(R')$
must have length at most one with $\rho(c(R'))=(\lfloor \frac{k}{2}
\rfloor, \lfloor \frac{k}{2} \rfloor-1, \dots, 1, 0, \dots, 0)$, where
there are $\lceil \frac{k}{2} \rceil$ $0$'s.   

Order the $k$ elements of $c(R')$ such that the $i^{\mathrm{th}}$ element
corresponds to the $i^{\mathrm{th}}$ entry of $\rho(c(R'))$. Afterwards, we can
order the $k$ equivalence classes of $R'$ correspondingly.  Thus we
get a unique ordered partition of $[n]$, so the inverse map is
well-defined. 
\end{proof}

\vspace{6pt}

\subsection{The number of nonisomorphic unlabeled $n$-element
  semiorders of length at most three} 
\begin{theorem}\label{openquestion}
For $n\geq 1$, we have
\begin{equation}\label{6.2}
f_{\leq 3}^n=\frac{3^{n-1}+1}{2}.
\end{equation}
\end{theorem}

\begin{corollary}
For $n\geq 2$, $f_{3}^n=3f_{3}^{n-1}+f_{\leq 2}^{n-2}-1 =
3f_{3}^{n-1}+f_{2}^{n-2}+f_{1}^{n-2}.$ 
\end{corollary}
\begin{proof}
By Theorem \ref{5.1}, $f_{\leq 2}^n=3f_{\leq 2}^{n-1}-f_{\leq
  2}^{n-2}.$  By equation (\ref{6.2}), $f_{\leq 3}^n=3f_{\leq
  3}^{n-1}-1.$ 
Therefore,
\begin{align*}
f_{3}^n=f_{\leq 3}^n-f_{\leq 2}^n&=3f_{\leq 3}^{n-1}-1-(3f_{\leq
  2}^{n-1}-f_{\leq 2}^{n-2})\\ 
&=3(f_{\leq 3}^{n-1}-f_{\leq 2}^{n-1})+f_{\leq 2}^{n-2}-1\\
&=3f_{3}^{n-1}+f_{2}^{n-2}+f_{1}^{n-2}.
\end{align*}
\end{proof}

\begin{rmk}
Theorem \ref{openquestion} can be directly derived from equation
(\ref{general}), or from the recurrence formula in  Theorem
\ref{5.1}.  However, we believe that there should be a more
straightforward bijective proof going on, which leaves an open
question for this paper. 
\end{rmk}

\end{document}